\documentclass[12pt, reqno]{amsart}%
\usepackage{amsmath, amsthm, amscd, amsfonts, amssymb, graphicx, color}
\usepackage[bookmarksnumbered, colorlinks, plainpages]{hyperref}%
\usepackage{amsmath}%
\usepackage{amsfonts}%
\usepackage{amssymb}%
\usepackage{graphicx}
\providecommand{\U}[1]{\protect \rule{.1in}{.1in}}
\newtheorem{theorem}{Theorem}[section]

\newtheorem{proposition}[theorem]{Proposition}
\newtheorem{corollary}[theorem]{Corollary}
\theoremstyle{definition}
\newtheorem{definition}[theorem]{Definition}

\theoremstyle{remark}
\newtheorem{remark}[theorem]{Remark}
\numberwithin{equation}{section}
\begin{document}
\title[Slant helices in three dimensional Lie groups]{Slant helices in three dimensional Lie groups}
\author[O. Zeki Okuyucu, \.{I}.G\"{o}k, Y. Yay\i, N. Ekmekci]{O. Zeki Okuyucu$^{1}$$^{*}$, \.{I}.G\"{o}k$^{2}$, Y. Yayl\i$^{2}$\ and N.
Ekmekci$^{2}$}
\address{$^{1}$ Bilecik \c{S}eyh Edeabali University, Faculty of Sciences and Arts, Department of
Mathematics, 11210, Bilecik, Turkey.}
\email{osman.okuyucu@bilecik.edu.tr}
\address{$^{2}$ Ankara University, Faculty of Science, Department of Mathematics,
06100, Tando\~{g}an, Ankara, Turkey.}
\email{igok@science.ankara.edu.tr}
\email{yayli@science.ankara.edu.tr}
\email{nekmekci@science.ankara.edu.tr}
\date{Received: xxxxxx; Revised: yyyyyy; Accepted: zzzzzz. }
\date{\indent $^{*}$ Corresponding author}
\date{Received: xxxxxx; Revised: yyyyyy; Accepted: zzzzzz. }
\date{\indent $^{*}$ Corresponding author}

\begin{abstract}
In this paper, we define slant helices in three dimensional Lie Groups with a
bi-invariant metric and obtain a characterization of slant helices. Moreover,
we give some relations between slant helices and their involutes, spherical images.

\end{abstract}
\maketitle

\setcounter{page}{1}

\setcounter{page}{1}


\setcounter{page}{1}


\section{Introduction}

In differential geometry, we think that curves are geometric set of points of
loci. Curves theory is important workframe in the differential geometry
studies and we have a lot of special curves such as geodesics, circles,
Bertrand curves, circular helices, general helices, slant helices etc.
Characterizations of these special curves are heavily studied for a long time
and are still studied. We can see helical structures in nature and mechanic
tools. In the field of computer aided design and computer graphics, helices
can be used for the tool path description, the simulation of kinematic motion
or design of highways. Also we can see the helix curve or helical structure in
fractal geometry, for instance hyperhelices.\ In differential geometry; a
curve of constant slope or general helix in Euclidean 3-space $\mathbb{E}^{3}%
$, is defined by the property that its tangent vector field makes a constant
angle with a fixed straight line (the axis of the general helix). A classical
result stated by \textit{M. A. Lancret} in 1802 and first proved
by\textit{\ B. de Saint Venant} in 1845 (see \cite{lancret,struik} for
details) is: A necessary and sufficient condition that a curve be a general
helix is that the ratio of curvature to torsion is constant. If both of
$\varkappa$ and $\tau$ are non-zero constants then the curve is called as a
circular helix. It is known that a straight line and a circle are
degenerate-helix examples ($\varkappa=0$, if the curve is straight line and
$\tau=0$, if the curve is a circle).


The Lancret theorem was revisited and solved by Barros \cite{barros} in
$3$-dimensional real space forms by using killing vector fields along curves.
Also in the same spaceforms, a characterization of helices and Cornu spirals
is given by Arroyo, Barros and Garay in \cite{garay}.


The degenarete semi-Riemannian geometry of Lie group is studied by
\c{C}\"{o}ken and \c{C}ift\c{c}i \cite{coken}. Moreover, they obtanied a
naturally reductive homogeneous semi-Riemannian space using the Lie group.
Then \c{C}ift\c{c}i \cite{ciftci} defined general helices in three dimensional
Lie groups with a bi-invariant metric and obtained a generalization of
Lancret's theorem and gave a relation between the geodesics of the so-called
cylinders and general helices.


Recently, \textit{Izumiya and Takeuchi}, in \cite{izu}, have introduced the
concept of slant helix in Euclidean $3$-space.\ A slant helix in Euclidean
space $\mathbb{E}^{3}$ was defined by the property that its principal normal
vector field makes a constant angle with a fixed direction. Moreover, Izumiya
and Takeuchi showed that $\alpha$ is a slant helix if and only if the geodesic
curvature of spherical image of principal normal indicatrix $\left(  N\right)
$ of a space curve $\alpha$
\[
\sigma_{N}\left(  s\right)  =\left(  \frac{\varkappa^{2}}{\left(
\varkappa^{2}+\tau^{2}\right)  ^{3/2}}\left(  \frac{\tau}{\varkappa}\right)
^{\prime}\right)  \left(  s\right)
\]
is a constant function. In \cite{kula}; Kula and Yayli have studied spherical
images of a slant helix and showed that the spherical images of a slant helix
are spherical helices. In \cite{kula1}, the authors characterize slant helices
by certain differential equations verified for each one of spherical
indicatrix in Euclidean $3$-space. Ali and Lopez, in \cite{ali}, have studied
slant helix in Minkowski $3$-space. They showed that the spherical indicatrix
of a slant helix are helices in $\mathbb{E}_{1}^{3}$. Then Ali and Turgut
studied position vector of a time-like slant helix with respect to standard
frame of Minkowski space $\mathbb{E}_{1}^{3}$ in terms of Frenet equations
(see \cite{ahmad} for details). Also slant helices are used in some
applications in quaternion algebra (see \cite{gok,gok1} for details).


In this paper, first of all, we define slant helices in a three dimensional
Lie group $G$ with a bi-invariant metric as a curve $\alpha:I\subset
\mathbb{R\rightarrow}G$ whose normal vector field makes a constant angle with
a left invariant vector field (Definition \ref{3.1}). And then the main result
to this paper is given as (Theorem \ref{3.6}): A curve $\alpha:I\subset
\mathbb{R\rightarrow}G$ with the Frenet apparatus $\left \{  T,N,B,\varkappa
,\tau \right \}  $ is a slant helix if and only if%
\[
\frac{\varkappa(H^{2}+1)^{\frac{3}{2}}}{H^{\shortmid}}%
\]
is a constant function where $H$ is a harmonic curvature function of the curve
$\alpha$ (Definition \ref{3.2}).


Then we define the involutes and spherical image of a curve in three
dimensional Lie group $G$. Also we show that the spherical image of a slant
helix and the involutes of a slant helix are general helices. Finally, we give
characterization of a slant helix if $G$ are Abellian, $SO^{3}$ and $S^{3}$.


Note that three dimensional Lie groups admitting bi-invariant metrics are
$SO\left(  3\right)  ,SU^{2}$ and Abellian Lie groups. So we believe that
characterizations of slant curves in this study will be useful for curves
theory in Lie groups.

\section{Preliminaries}

Let $G$ be a Lie group with a bi-invariant metric $\left \langle \text{
},\right \rangle $ and $D$ be the Levi-Civita connection of Lie group $G.$ If
$\mathfrak{g}$ denotes the Lie algebra of $G$ then we know that $\mathfrak{g}
$ is issomorphic to $T_{e}G$ where $e$ is neutral element of $G.$ If
$\left \langle \text{ },\right \rangle $ is a bi-invariant metric on $G$ then we
have%
\begin{equation}
\left \langle X,\left[  Y,Z\right]  \right \rangle =\left \langle \left[
X,Y\right]  ,Z\right \rangle \label{2-1}%
\end{equation}
and
\begin{equation}
D_{X}Y=\frac{1}{2}\left[  X,Y\right] \label{2-2}%
\end{equation}
for all $X,Y$ and $Z\in \mathfrak{g}.$


Let $\alpha:I\subset \mathbb{R\rightarrow}G$ be an arc-lenghted curve and
$\left \{  X_{1},X_{2,}...,X_{n}\right \}  $ be an orthonormal basis of
$\mathfrak{g}.$ In this case, we write that any two vector fields $W$ and $Z$
along the curve $\alpha \ $as $W=\sum_{i=1}^{n}w_{i}X_{i}$ and $Z=\sum
_{i=1}^{n}z_{i}X_{i}$ where $w_{i}:I\rightarrow \mathbb{R}$ and $z_{i}%
:I\rightarrow \mathbb{R}$ are smooth functions. Also the Lie bracket of two
vector fields $W$ and $Z$ is given
\[
\left[  W,Z\right]  =\sum_{i=1}^{n}w_{i}z_{i}\left[  X_{i},X_{j}\right]
\]
and the covariant derivative of $W$ along the curve $\alpha$ with the notation
$D_{\alpha^{\shortmid}}W$ is given as follows%
\begin{equation}
D_{\alpha^{\shortmid}}W=\overset{\cdot}{W}+\frac{1}{2}\left[  T,W\right]
\label{2-3}%
\end{equation}
where $T=\alpha^{\prime}$ and $\overset{\cdot}{W}=\sum_{i=1}^{n}\overset
{\cdot}{w_{i}}X_{i}$ or $\overset{\cdot}{W}=\sum_{i=1}^{n}\frac{dw}{dt}X_{i}.$
Note that if $W$ is the left-invariant vector field to the curve $\alpha$ then
$\overset{\cdot}{W}=0$ (see \cite{crouch} for details).


Let $G$ be a three dimensional Lie group and $\left(  T,N,B,\varkappa
,\tau \right)  $ denote the Frenet apparatus of the curve $\alpha$, and
calculate $\varkappa=\overset{\cdot}{\left \Vert T\right \Vert }.$


\begin{definition}
\label{2.1}Let $\alpha:I\subset \mathbb{R\rightarrow}G$ be a parametrized
curve. Then $\alpha$ is called a general helix if it makes a constant angle
with a left-invariant vector field $X$. That is,%
\[
\left \langle T(s),X\right \rangle =\cos \theta \text{ for all }s\in I,
\]
for the left-invariant vector field $X\in g$ is unit length and $\theta$ is a
constant angle between $X$ and $T$ which is the tangent vector field of the
curve $\alpha$ (see \cite{ciftci}).
\end{definition}


\begin{definition}
\label{2.2}Let $\alpha:I\subset \mathbb{R\rightarrow}G$ be a parametrized curve
with the Frenet apparatus $\left(  T,N,B,\varkappa,\tau \right)  $ then
\begin{equation}
\tau_{G}=\frac{1}{2}\left \langle \left[  T,N\right]  ,B\right \rangle
\label{2-4}%
\end{equation}
or
\[
\tau_{G}=\frac{1}{2\varkappa^{2}\tau}\overset{\cdot \cdot \text{
\  \  \  \  \  \  \  \ }\cdot}{\left \langle T,\left[  T,T\right]  \right \rangle
}+\frac{1}{4\varkappa^{2}\tau}\overset{\text{ \  \ }\cdot}{\left \Vert \left[
T,T\right]  \right \Vert ^{2}}%
\]
(see \cite{ciftci}).
\end{definition}


\begin{theorem}
\label{2.3}Let $\alpha:I\subset \mathbb{R\rightarrow}G$ be a parametrized curve
with the Frenet apparatus $\left(  T,N,B,\varkappa,\tau \right)  $. If the
curve $\alpha$ is a general helix, if and only if,%
\[
\tau=c\varkappa+\tau_{G}%
\]
where c is a constant (see \cite{ciftci}).
\end{theorem}

\section{Slant helices in a three dimensional Lie group}

In this section we define slant helix and its axis in a three dimensional Lie
group $G$ with a bi-invariant metric $\left \langle \text{ },\right \rangle $.
Also we give a characterization and some characterizations of the slant
helices in the special cases of $G$.


\begin{definition}
\label{3.1}Let $\alpha:I\subset \mathbb{R\rightarrow}G$ be an arc length
parametrized curve. Then $\alpha$ is called a slant helix if its principal
normal vector makes a constant angle with a left-invariant vector field $X$
which is unit length. That is,%
\[
\left \langle N(s),X\right \rangle =\cos \theta \text{ for all }s\in I,
\]
where $\theta \neq \frac{\pi}{2}$ is a constant angle between $X$ and $N$ which
is the principal normal vector field of the curve $\alpha$.
\end{definition}


\begin{definition}
\label{3.2}Let $\alpha:I\subset \mathbb{R\rightarrow}G$ be an arc length
parametrized curve with the Frenet apparatus $\left \{  T,N,B,\varkappa
,\tau \right \}  .$ Then the harmonic curvature function of the curve $\alpha$
is defined by%
\[
H=\dfrac{\tau-\tau_{G}}{\varkappa}%
\]
where $\tau_{G}=\frac{1}{2}\left \langle \left[  T,N\right]  ,B\right \rangle .
$
\end{definition}


\begin{definition}
\label{3.3}Let $\alpha:I\subset \mathbb{R\rightarrow}G$ be an arc length
parametrized curve with the Frenet apparatus $\left \{  T,N,B,\varkappa
,\tau \right \}  $. Then the geodesic curvature of the spherical image of the
principal normal indicatrix $\left(  N\right)  $ of the curve $\alpha$ is
defined by a constant $\sigma_{N}$ given by%
\[
\sigma_{N}=\frac{\varkappa(1+H^{2})^{\frac{3}{2}}}{H^{\shortmid}}%
\]
where $H$ is harmonic curvature function of the curve $\alpha.$
\end{definition}


\begin{proposition}
\label{3.4}Let $\alpha:I\subset \mathbb{R\rightarrow}G$ be an arc length
parametrized curve with the Frenet apparatus $\left \{  T,N,B\right \}  $. Then
the following equalities%
\begin{align*}
\left[  T,N\right]   &  =\left \langle \left[  T,N\right]  ,B\right \rangle
B=2\tau_{G}B\\
\left[  T,B\right]   &  =\left \langle \left[  T,B\right]  ,N\right \rangle
N=-2\tau_{G}B
\end{align*}
hold.
\end{proposition}


\begin{proof}
Let $\alpha:I\subset \mathbb{R\rightarrow}G$ be an arc length parametrized
curve with the Frenet apparatus $\left \{  T,N,B\right \}  $. Since $\left[
T,N\right]  \in Sp\left \{  T,N,B\right \}  ,$ we can write%
\begin{equation}
\left[  T,N\right]  =\lambda_{1}T+\lambda_{2}N+\lambda_{3}B.\label{3-1}%
\end{equation}
If we multiply the two sides of the Eq. \eqref{3-1} with $T,$ $N$ and $B,$
respectively%
\begin{align*}
\left \langle \left[  T,N\right]  ,T\right \rangle  &  =\lambda_{1}=0,\\
\left \langle \left[  T,N\right]  ,N\right \rangle  &  =\lambda_{2}=0,\\
\left \langle \left[  T,N\right]  ,B\right \rangle  &  =\lambda_{3}.
\end{align*}
Thus we can write%
\[
\left[  T,N\right]  =\left \langle \left[  T,N\right]  ,B\right \rangle B,
\]
or using the Eq. \eqref{2-4} and the last equation, we get%
\[
\left[  T,N\right]  =2\tau_{G}B.
\]
On the other hand, using a similar method we can easily show that%
\[
\left[  T,B\right]  =-2\tau_{G}N.
\]
Which complete the proof.
\end{proof}


\begin{proposition}
\label{3.5}Let $\alpha:I\subset \mathbb{R\rightarrow}G$ be a parametrized curve
with arc length parameter s and $\left \{  T,N,B\right \}  $ denote the Frenet
frame of the curve $\alpha$. If the curve $\alpha$ is a slant helix in $G$,
then the axis of $\alpha$ is%
\[
X=\left \{  \frac{\varkappa H\left(  1+H^{2}\right)  }{H^{\shortmid}}%
T+N+\frac{\varkappa \left(  1+H^{2}\right)  }{H^{\shortmid}}B\right \}
\cos \theta
\]
where $H=\dfrac{\tau-\tau_{G}}{\varkappa}$ is harmonic curvature function of
the curve $\alpha$ and $\theta \neq \frac{\pi}{2}$ is a constant angle.
\end{proposition}


\begin{proof}
If the axis of slant helix $\alpha$ is $X$, then we can write%
\[
X=\lambda_{1}T+\lambda_{2}N+\lambda_{3}B
\]
where $\lambda_{1}=\left \langle T,X\right \rangle ,$ $\lambda_{2}=\left \langle
N,X\right \rangle $ and $\lambda_{3}=\left \langle B,X\right \rangle .$

And we know from the Definition \ref{3.1} that%
\begin{equation}
\left \langle N(s),X\right \rangle =\cos \theta \text{ for all }s\in I,\label{3-2}%
\end{equation}
where the left-invariant vector field $X\in \mathfrak{g}$ is unit length and
$\theta$ is a constant angle between $X$ and $N$ which is the principal normal
vector field of the curve $\alpha$. By differentiating $\left \langle
N(s),X\right \rangle =\cos \theta,$ we get%
\[
\left \langle D_{T}N,X\right \rangle +\left \langle N,D_{T}X\right \rangle =0,
\]
or using the Eq. \eqref{2-3} and the Frenet formulas%
\[
-\kappa \left \langle T,X\right \rangle +\tau \left \langle B,X\right \rangle
-\dfrac{1}{2}\left \langle \left[  T,N\right]  ,X\right \rangle =0,
\]
and with the help of the Proposition \ref{3.4}, we get%
\begin{equation}
\left \langle T,X\right \rangle =H\left \langle B,X\right \rangle ,\label{3-3}%
\end{equation}
where $H=\dfrac{\tau-\tau_{G}}{\varkappa}$ is harmonic curvature function of
the curve $\alpha$.

Again differentiating the Eq. \eqref{3-3}, we have%
\[
\left \langle D_{T}T,X\right \rangle +\left \langle T,D_{T}X\right \rangle
=H^{\shortmid}\left \langle B,X\right \rangle +H\left \{  \left \langle
D_{T}B,X\right \rangle +\left \langle B,D_{T}X\right \rangle \right \}
\]
then by using the Eq. \eqref{2-3} and the Proposition \ref{3.4} we obtain%
\begin{equation}
\left \langle B,X\right \rangle =\frac{\varkappa \left(  1+H^{2}\right)
}{H^{\shortmid}}\left \langle N,X\right \rangle .\label{3-4}%
\end{equation}
Then if we write the Eq. \eqref{3-4} in the Eq. \eqref{3-3}, we get%
\begin{equation}
\left \langle T,X\right \rangle =\frac{\varkappa H}{H^{\shortmid}}\left(
1+H^{2}\right)  \left \langle N,X\right \rangle .\label{3-5}%
\end{equation}
Consequently, using the equations \eqref{3-2}, \eqref{3-4} and \eqref{3-5} the
axis of slant helix $\alpha$ is given by
\[
X=\left \{  \frac{\varkappa H\left(  1+H^{2}\right)  }{H^{\shortmid}}%
T+N+\frac{\varkappa \left(  1+H^{2}\right)  }{H^{\shortmid}}B\right \}
\cos \theta,
\]
which completes the proof.
\end{proof}


\begin{theorem}
\label{3.6}Let $\alpha:I\subset \mathbb{R\rightarrow}G$ \ be a unit speed curve
with the Frenet apparatus $\left(  T,N,B,\varkappa,\tau \right)  $. Then
$\alpha$ is a slant helix if and only if%
\[
\sigma_{N}=\frac{\varkappa(1+H^{2})^{\frac{3}{2}}}{H^{\shortmid}}=\tan \theta
\]
is a constant where $H$ is a harmonic curvature function of the curve $\alpha$
and $\theta \neq \frac{\pi}{2}$ is a constant.
\end{theorem}


\begin{proof}
If the axis of slant helix $\alpha$ is $X$, then using the Proposition
\ref{3.5} we have%
\[
X=\left \{  \frac{\varkappa H\left(  1+H^{2}\right)  }{H^{\shortmid}}%
T+N+\frac{\varkappa \left(  1+H^{2}\right)  }{H^{\shortmid}}B\right \}
\cos \theta.
\]
Since $X$ is unit lenght vector field then we can easily see that%
\[
\frac{\varkappa(H^{2}+1)^{\frac{3}{2}}}{H^{\shortmid}}=\tan \theta
\]
is a constant.

Conversely, if $\sigma_{N}\left(  s\right)  $ is constant then the result is
obvious. This complete the proof.
\end{proof}


In the following remark, we note that three dimensional Lie groups admitting
bi-invariant metrics are $S^{3},$ $SO^{3}$ and Abelian Lie groups using the
same notation as in \cite{ciftci} and \cite{santo} as follows:


\begin{remark}
\label{3.7}Let $G$ be a Lie group with a bi-invariant metric $\left \langle
\text{ },\right \rangle $. Then the following equalities can be given in
different Lie groups.
\end{remark}

$i$ ) If $G$ is abelian group then $\tau_{G}=0.$

$ii)$ If $G$ is $SO^{3}$ then $\tau_{G}=\frac{1}{2}$.

$iii)$ If $G$ is $SU^{2}$ then $\tau_{G}=1$

(see for details \cite{ciftci} and \cite{santo}).


\begin{corollary}
\label{3.8}Let $\alpha$ be a unit speed curve with the Frenet apparatus
$\left \{  T,N,B\right \}  $ in the Abellian Lie group $G$. Then $\alpha$ is a
slant helix if and only if%
\[
\sigma_{N}=\frac{\left(  \varkappa^{2}+\tau^{2}\right)  ^{3/2}}{\varkappa
^{2}\left(  \dfrac{\tau}{\varkappa}\right)  ^{\shortmid}}%
\]
is a constant function.
\end{corollary}


\begin{proof}
If $G$ is Abellian Lie group then using the above Remark and the Theorem
\ref{3.6} we have the result.
\end{proof}


So, the above Corollary shows that the study is a generalization of slant
helices defined by Izimuya \cite{izu} in Euclidean 3-space. Moreover, with a
similar proof, we have the following two corollaries.


\begin{corollary}
\label{3.9}Let $\alpha$ be unit speed curve with the Frenet apparatus
$\left \{  T,N,B\right \}  $ in the Lie group $SU^{2}$. Then $\alpha$ is a slant
helix if and only if%
\[
\sigma_{N}=\frac{\left(  \varkappa^{2}+\left(  \tau-1\right)  ^{2}\right)
^{3/2}}{\varkappa^{2}\left(  \dfrac{\tau-1}{\varkappa}\right)  ^{\shortmid}}%
\]
is a constant function.
\end{corollary}


\begin{corollary}
\label{3.10}Let $\alpha$ be unit speed curve with the Frenet apparatus
$\left \{  T,N,B\right \}  $ in the Lie group $SO^{3}$. Then $\alpha$ is a slant
helix if and only if%
\[
\sigma_{N}=\frac{\left(  \varkappa^{2}+\left(  \tau-\frac{1}{2}\right)
^{2}\right)  ^{3/2}}{\varkappa^{2}\left(  \dfrac{\tau-\frac{1}{2}}{\varkappa
}\right)  ^{\shortmid}}%
\]
is a constant function.
\end{corollary}


\section{Spherical Images of Slant Helices in the three dimensional Lie group}

In Euclidean geometry, the spherical indicatrix of a space curve is defined as
follows: Let $\alpha$ be a unit speed regular curve in Euclidean $3$-space
with Frenet vectors $t$ , $n$ and $b$. The unit tangent vectors along the
curve $\alpha$ generate a curve $\alpha_{T}$ on the sphere of radius 1 about
the origin. The curve $\alpha_{T}$ is called the spherical indicatrix of $t$
or more commonly, $\alpha_{T}$ is called tangent indicatrix of the curve
$\alpha$. If $\alpha=\alpha(s)$ is a natural representation of $\alpha$, then
$\alpha_{T}=T(s)$ will be a representation of $\alpha_{T}$. Similarly one
considers the principal normal indicatrix $\alpha_{N}=N(s)$ and binormal
indicatrix $\alpha_{B}=B(s)$. It is clear that, this definition is related
with the spherical curve \cite{struik}.


In this section, firstly we define spherical indicatrices of slant helices
with the help of the studies \cite{noakes,ripol} and then investigate the
relation between slant helices and their spherical indicatrices in
3-dimensional Lie group. Morever, we give some theorems with their proofs and
some examples in special Lie groups.

\subsection{Tangent indicatrices of slant helices:}

\begin{definition}
\label{4.1}Let $\alpha:I\subset \mathbb{R\rightarrow}G$ \ be an arc-lenghted
regular curve. Its tangent indicatrix is the parametrized curve $\beta
:I\subset \mathbb{R\rightarrow}S^{2}\subset \mathfrak{g}$ defined by%
\[
\beta \left(  s^{\ast}\right)  =T(s)=\sum_{i=1}^{3}\text{ }t_{i}X_{i}\text{ for
all }s\in I
\]
where $\left \{  X_{1},X_{2},X_{3}\right \}  $ is an orthonormal basis of
$\mathfrak{g}$ and $s^{\ast}$ is the arc length parameter of $\beta.$
\end{definition}


\begin{theorem}
\label{4.2}Let $\alpha$\ be an arc-lenghted regular curve and $\beta$ be the
tangent indicatrix of the curve $\alpha.$ Then the curve $\alpha$ is a slant
helix in three dimensional Lie group $G$ if and only if \ the curve $\beta$ is
a general helix on $S^{2}$.
\end{theorem}


\begin{proof}
We assume that the curve $\alpha$ is a slant helix in a three dimensional Lie
group and $\alpha_{T}$ is the tangent indicatrix of the curve $\alpha.$ From
the Definition \ref{4.1} we get%
\[
\beta \left(  s^{\ast}\right)  =T(s)
\]
then differentiating the last equation and using the Eq. \eqref{2-3}, we have%
\begin{align*}
\frac{d\beta}{ds^{\ast}}\frac{ds^{\ast}}{ds} &  =\overset{\cdot}{T}%
=D_{T}T-\dfrac{1}{2}\left[  T,T\right] \\
\frac{d\beta}{ds^{\ast}}\frac{ds^{\ast}}{ds} &  =\varkappa N.
\end{align*}
Then assuming that $\varkappa \rangle0$ we obtain%
\begin{equation}
\frac{ds^{\ast}}{ds}=\varkappa \label{4-1}%
\end{equation}
and%
\begin{equation}
T_{\beta}\left(  s^{\ast}\right)  =N(s).\label{4-2}%
\end{equation}
If we differentiate the last equation and use Frenet formulas then we obtain%
\begin{align*}
\varkappa_{\beta}N_{\beta}\left(  s^{\ast}\right)  \frac{ds^{\ast}}{ds} &
=\overset{\cdot}{N}=D_{T}N-\dfrac{1}{2}\left[  T,N\right] \\
\varkappa_{\beta}N_{\beta}\left(  s^{\ast}\right)  \varkappa &  =-\kappa
T+\tau B-\dfrac{1}{2}\left \langle \left[  T,N\right]  ,B\right \rangle B
\end{align*}
or with the help of the Proposition \ref{3.4}, we get%
\[
\varkappa_{\beta}N_{\beta}\left(  s^{\ast}\right)  =-T+HB
\]

where $\varkappa_{\beta}$ is the curvature of $\beta.$ Hence%
\[
\varkappa_{\beta}=\sqrt{1+H^{2}}%
\]
and%
\begin{equation}
N_{\beta}\left(  s^{\ast}\right)  =-\tfrac{1}{\sqrt{1+H^{2}}}T+\tfrac{H}%
{\sqrt{1+H^{2}}}B\label{4-3}%
\end{equation}
Then using the Eq.\eqref{4-2} and the Eq.\eqref{4-3} we have%
\begin{align}
B_{\beta}\left(  s^{\ast}\right)   & =T_{\beta}\left(  s^{\ast}\right)  \times
N_{\beta}\left(  s^{\ast}\right) \nonumber \\
& =\tfrac{H}{\sqrt{1+H^{2}}}T+\tfrac{1}{\sqrt{1+H^{2}}}B.\label{4-4}%
\end{align}
Using the differentiation of the last equation and the Proposition \ref{3.4},
this implies%
\[
\left(  \tau_{\beta}-\tau_{G_{\beta}}\right)  N_{\beta}\left(  s^{\ast
}\right)  \frac{ds^{\ast}}{ds}=-\tfrac{H^{\prime}}{\left(  1+H^{2}\right)
^{3/2}}T+\tfrac{HH^{\prime}}{\left(  1+H^{2}\right)  ^{3/2}}B
\]
or using the Eq.\eqref{4-1}, we have%
\[
\left(  \tau_{\beta}-\tau_{G_{\beta}}\right)  N_{\beta}\left(  s^{\ast
}\right)  =-\tfrac{H^{\prime}}{\varkappa \left(  1+H^{2}\right)  ^{3/2}%
}T+\tfrac{HH^{\prime}}{\varkappa \left(  1+H^{2}\right)  ^{3/2}}B
\]
where $\tau_{G_{\beta}}=\frac{1}{2}\left \langle \left[  T_{\beta},N_{\beta
}\right]  ,B_{\beta}\right \rangle .$ Thus we compute%
\[
\tau_{\beta}=\frac{H^{\shortmid}}{\varkappa \left(  1+H^{2}\right)  }%
+\tau_{G_{\beta}}%
\]
where $\tau_{\beta}$ is the torsion of $\beta.$ The we can easily see that
$\tfrac{\tau_{\beta}-\tau_{G_{\beta}}}{\varkappa_{\beta}}=\frac{H^{\shortmid}%
}{\varkappa \left(  1+H^{2}\right)  ^{3/2}}$\ is a constant function. In other
words, using the Theorem \ref{2.3} we can easily obtain that $\beta$ is a
general helix.

Conversely, we assume that $\beta$ is a general helix then we can easily see
that $\alpha$ is a slant helix. These complete the proof.
\end{proof}


\begin{corollary}
\label{4.3}Let $\alpha$\ be an arc-lenghted regular curve with the Frenet
vector fields $\left \{  T,N,B\right \}  $ in the Lie group $G$ and $\beta$ be
the tangent indicatrix of the curve $\alpha.$ Then $\tau_{G_{\beta}}=\tau_{G}$
for the curves $\alpha$ and $\beta.$
\end{corollary}

\begin{proof}
It is obvious using the equations \eqref{4-2}, \eqref{4-3} and \eqref{4-4}.
\end{proof}

\subsection{Normal indicatrices of slant helices:}

\begin{definition}
\label{4.4}Let $\alpha:I\subset \mathbb{R\rightarrow}G$ \ be an arc-lenghted
regular curve. Its normal indicatrix is the parametrized curve $\gamma
:I\subset \mathbb{R\rightarrow}S^{2}\subset \mathfrak{g}$ defined by%
\[
\gamma \left(  s^{\ast}\right)  =N(s)=\sum_{i=1}^{3}\text{ }n_{i}X_{i}\text{
for all }s\in I
\]
where $\left \{  X_{1},X_{2},X_{3}\right \}  $ is an orthonormal basis of
$\mathfrak{g}$ and $s^{\ast}$ is the arc length parameter of $\gamma.$
\end{definition}


\begin{theorem}
\label{4.5}Let $\alpha$\ be an arc-lenghted slant helix in three dimensional
Lie Group $G$ and $\gamma$ be the normal indicatrix of the curve $\alpha.$
Then the curve $\gamma$ is a plane curve on $S^{2}$.
\end{theorem}


\begin{proof}
We assume that the curve $\alpha$ is a slant helix in a three dimensional Lie
group and $\gamma$ is the normal indicatrix of the curve $\alpha.$ From the
Definition \ref{4.4} we get%
\begin{equation}
\gamma \left(  s^{\ast}\right)  =N(s).\label{4-5}%
\end{equation}
Then differentiating the Eq. \eqref{4-5} and using the Eq. \eqref{2-3} we have%
\begin{align*}
\frac{d\gamma}{ds^{\ast}}\frac{ds^{\ast}}{ds} &  =\overset{\cdot}{N}%
=D_{T}N-\dfrac{1}{2}\left[  T,N\right] \\
&  =-\varkappa T+\tau B-\frac{1}{2}\left \langle \left[  T,N\right]
,B\right \rangle B\\
&  =-\varkappa T+\left(  \tau-\tau_{G}\right)  B\\
&  =-\varkappa T+\varkappa HB
\end{align*}
Then assuming that $\varkappa \rangle0$ we obtain%
\begin{equation}
\frac{ds^{\ast}}{ds}=\varkappa \sqrt{1+H^{2}}\label{4-6}%
\end{equation}
and%
\[
\frac{d\gamma}{ds^{\ast}}=\frac{1}{\sqrt{1+H^{2}}}\left(  -T+HB\right)  .
\]
If we differentiate the last equation, then we obtain%
\begin{align*}
\frac{d^{2}\gamma}{ds^{\ast^{2}}}\frac{ds^{\ast}}{ds} &  =-\dfrac
{HH^{\shortmid}}{\left(  1+H^{2}\right)  ^{3/2}}\left(  -T+HB\right)
+\frac{1}{\sqrt{1+H^{2}}}\left(  -\overset{\cdot}{T}+H^{\shortmid}%
B+H\overset{\cdot}{B}\right) \\
&  =-\dfrac{HH^{\shortmid}}{\left(  1+H^{2}\right)  ^{3/2}}\left(
-T+HB\right)  +\frac{1}{\sqrt{1+H^{2}}}\left \{  -\varkappa N+H^{\shortmid
}B+H\left(  -\tau N-\dfrac{1}{2}\left[  T,B\right]  \right)  \right \}
\end{align*}
and by using the Eq. \eqref{4-6} with together Proposition \ref{3.4} we obtain%
\begin{align*}
\frac{d^{2}\gamma}{ds^{\ast^{2}}} &  =-\frac{H}{\sqrt{1+H^{2}}}\dfrac
{H^{\shortmid}}{\varkappa \left(  1+H^{2}\right)  ^{3/2}}\left(  -T+HB\right)
+\dfrac{1}{\varkappa \left(  1+H^{2}\right)  }\left \{  \left(  -\varkappa
-H\left(  \tau-\tau_{G}\right)  \right)  N+H^{\shortmid}B\right \} \\
&  =-\frac{H}{\sqrt{1+H^{2}}}\dfrac{H^{\shortmid}}{\varkappa \left(
1+H^{2}\right)  ^{3/2}}\left(  -T+HB\right)  +\dfrac{1}{\varkappa \left(
1+H^{2}\right)  }\left \{  -\varkappa \left(  1+H^{2}\right)  N+H^{\shortmid
}B\right \}  .
\end{align*}
Since $\alpha$ is a slant helix, $\sigma_{N}(s)$ is a constant function. So,
we can obtain%
\begin{equation}
\frac{d^{2}\gamma}{ds^{\ast^{2}}}=\frac{1}{\sigma_{N}(s)}\frac{H}%
{\sqrt{1+H^{2}}}T-N+\frac{1}{\sigma_{N}(s)}\frac{1}{\sqrt{1+H^{2}}%
}B\label{4-7}%
\end{equation}
Hence%
\[
\varkappa_{\gamma}=\left \Vert \frac{d^{2}\gamma}{ds^{\ast^{2}}}\right \Vert
=\frac{1}{\left \vert \sigma_{N}\right \vert }\sqrt{1+\sigma_{N}^{2}}%
\]
where $\varkappa_{\gamma}$ is the curvature of $\gamma$ . Then differentiating
the Eq. \eqref{4-7} and using the Definition \ref{3.3} we have%
\begin{align*}
\frac{d^{3}\gamma}{ds^{\ast^{3}}}\varkappa \sqrt{1+H^{2}} &  =-\frac{1}%
{\sigma_{N}}\left \{  \dfrac{H^{\shortmid}}{\left(  1+H^{2}\right)  ^{3/2}%
}\left(  -T+HB\right)  +\frac{H}{\sqrt{1+H^{2}}}\left(  -\overset{\cdot}%
{T}+H^{\shortmid}B+H\overset{\cdot}{B}\right)  \right \}  -\overset{\cdot}{N}\\
&  +\frac{1}{\sigma_{N}}\left(  \frac{HH^{\shortmid}}{\sqrt{1+H^{2}}}%
B+\sqrt{1+H^{2}}\overset{\cdot}{B}\right) \\
&  =-\frac{1}{\sigma_{N}}\left \{  \dfrac{H^{\shortmid}}{\left(  1+H^{2}%
\right)  ^{3/2}}\left(  -T+HB\right)  +\frac{H}{\sqrt{1+H^{2}}}\left(
-\varkappa \left(  1+H^{2}\right)  N+H^{\shortmid}B\right)  \right \} \\
&  -D_{T}N+\dfrac{1}{2}\left[  T,N\right]  +\frac{1}{\sigma_{N}}\left(
\frac{HH^{\shortmid}}{\sqrt{1+H^{2}}}B+\sqrt{1+H^{2}}\left(  D_{T}B-\dfrac
{1}{2}\left[  T,B\right]  \right)  \right)
\end{align*}
then by using the Proposition \ref{3.4}, we obtain%
\begin{equation}
\frac{d^{3}\gamma}{ds^{\ast^{3}}}=\varkappa \frac{\sigma_{N}^{2}+1}{\sigma
_{N}^{2}}T-\varkappa H\frac{\sigma_{N}^{2}+1}{\sigma_{N}^{2}}B\label{4-8}%
\end{equation}
Thus we compute%
\[
\tau_{\gamma}=\frac{\det \left(  \gamma^{\shortmid},\gamma^{\shortparallel
},\gamma^{\shortmid \shortmid \shortmid}\right)  }{\left \Vert \gamma^{\shortmid
}\times \gamma^{\shortparallel}\right \Vert ^{2}}=0
\]
where $\tau_{\gamma}$ is the torsion of $\gamma.$ Hence $\gamma$ is a plane
curve. This complete the proof.
\end{proof}

\subsection{Binormal indicatrices of slant helices:}

\begin{definition}
\label{4.6}Let $\alpha:I\subset \mathbb{R\rightarrow}G$ \ be an arc-lenghted
regular curve. Its binormal indicatrix is the parametrized curve
$\delta:I\subset \mathbb{R\rightarrow}S^{2}\subset \mathfrak{g}$ defined by as%
\[
\delta \left(  s^{\ast}\right)  =B(s)=\sum_{i=1}^{3}\text{ }b_{i}X_{i}\text{
for all }s\in I
\]
where $\left \{  X_{1},X_{2},X_{3}\right \}  $ is an orthonormal basis of
$\mathfrak{g}$ and $s^{\ast}$ is the arc length parameter of $\delta.$
\end{definition}


\begin{theorem}
\label{4.7}Let $\alpha$\ be an arc-lenghted regular curve and $\gamma$ be the
binormal indicatrix of the curve $\alpha.$ Then the curve $\alpha$ is a slant
helix in three dimensional Lie group $G$ if and only if \ the curve $\delta$
is a general helix on $S^{2}$.
\end{theorem}


\begin{proof}
We assume that $\alpha$ be a slant helix in a three dimensional Lie group and
$\alpha_{B}$ be the tangent indicatrix of the curve $\alpha.$ From the
Definition \ref{4.6} we get%
\begin{equation}
\delta \left(  s^{\ast}\right)  =B(s)\label{4-9}%
\end{equation}
then differentiating the Eq.\eqref{4-9} and using the Eq.\eqref{2-3}, we have%
\begin{align*}
\frac{d\delta}{ds^{\ast}}\frac{ds^{\ast}}{ds} &  =\overset{\cdot}{B}%
=D_{T}B-\dfrac{1}{2}\left[  T,B\right] \\
\frac{d\delta}{ds^{\ast}}\frac{ds^{\ast}}{ds} &  =-\varkappa HN.
\end{align*}
Then assuming that $\varepsilon=\left \{
\begin{array}
[c]{cc}%
1 & \text{ },\text{ if }\varkappa H\rangle0\\
-1 & \text{ },\text{ if }\varkappa H\langle0
\end{array}
\right \}  $ we have%
\[
\frac{ds^{\ast}}{ds}=\varepsilon \varkappa H
\]
and%
\begin{equation}
T_{\delta}\left(  s^{\ast}\right)  =-\varepsilon N(s).\label{4-10}%
\end{equation}
If we differentiate the last equation then we obtain%
\begin{align*}
\varkappa_{\delta}N_{\delta}\left(  s^{\ast}\right)  \frac{ds^{\ast}}{ds} &
=-\varepsilon \overset{\cdot}{N}=-\varepsilon D_{T}N+\varepsilon \dfrac{1}%
{2}\left[  T,N\right] \\
\varkappa_{\delta}N_{\delta}\left(  s^{\ast}\right)  \frac{ds^{\ast}}{ds} &
=\varepsilon \varkappa T-\varepsilon \tau B+\varepsilon \dfrac{1}{2}\left \langle
\left[  T,N\right]  ,B\right \rangle B\\
\varkappa_{\delta}N_{\delta}\left(  s^{\ast}\right)  \varepsilon \varkappa H &
=\varepsilon \varkappa T-\varepsilon(\tau-\tau_{G})B\\
\varkappa_{\delta}N_{\delta}\left(  s^{\ast}\right)   &  =\frac{1}{H}T-B
\end{align*}
where $\varkappa_{\delta}$ is the curvature of $\delta.$ Hence%
\[
\varkappa_{\delta}=\frac{1}{\left \vert H\right \vert }\sqrt{1+H^{2}}%
\]
and assuming that $\varkappa \rangle0$ we have
\begin{equation}
N_{\delta}\left(  s^{\ast}\right)  =\tfrac{\varepsilon}{\sqrt{1+H^{2}}%
}T-\tfrac{\varepsilon H}{\sqrt{1+H^{2}}}B\label{4-11}%
\end{equation}
Then using the Eq.\eqref{4-10} and the Eq.\eqref{4-11} we have%
\begin{align}
B_{\delta}\left(  s^{\ast}\right)   & =T_{\delta}\left(  s^{\ast}\right)
\times N_{\delta}\left(  s^{\ast}\right) \nonumber \\
& =-\tfrac{H}{\sqrt{1+H^{2}}}T+\tfrac{1}{\sqrt{1+H^{2}}}B.\label{4-12}%
\end{align}

Using the differentiation of the last equation and the Proposition \ref{3.4},
this implies%
\[
\left(  \tau_{\delta}-\tau_{G_{\delta}}\right)  N_{\delta}\left(  s^{\ast
}\right)  \frac{ds^{\ast}}{ds}=\tfrac{H^{\prime}}{\left(  1+H^{2}\right)
^{3/2}}T+\tfrac{HH^{\prime}}{\left(  1+H^{2}\right)  ^{3/2}}B
\]
or using the equality $\frac{ds^{\ast}}{ds}=\varepsilon \varkappa H$, we have%
\[
\left(  \tau_{\delta}-\tau_{G_{\delta}}\right)  N_{\delta}\left(  s^{\ast
}\right)  =\tfrac{H^{\prime}}{\varkappa \left(  1+H^{2}\right)  ^{3/2}}%
T+\tfrac{HH^{\prime}}{\varkappa \left(  1+H^{2}\right)  ^{3/2}}B
\]
where $\tau_{G_{\delta}}=\frac{1}{2}\left \langle \left[  T_{\delta},N_{\delta
}\right]  ,B_{\delta}\right \rangle .$ Thus we have%
\[
\tau_{\delta}=\frac{H^{\shortmid}}{\varkappa H\left(  1+H^{2}\right)  }%
+\tau_{G_{\delta}}%
\]
where $\tau_{\delta}$ is the torsion of $\delta$ and so $\dfrac{\tau_{\delta
}-\tau_{G_{\delta}}}{\varkappa_{\delta}}=\dfrac{H^{\shortmid}}{\varkappa
\left(  1+H^{2}\right)  ^{3/2}}$\ is a constant function, that is $\delta$ is
a general helix.

Conversely, we assume that $\delta$ is a general helix then we can see easily
that $\alpha$ is a slant helix. These complete the proof.
\end{proof}


\begin{corollary}
Let $\alpha$\ be an arc-lenghted regular curve with the Frenet vector fields
$\left \{  T,N,B\right \}  $ in the Lie group $G$ and $\delta$ be the binormal
indicatrix of the curve $\alpha.$ Then $\tau_{G_{\delta}}=\tau_{G}$ for the
curves $\alpha$ and $\delta.$
\end{corollary}

\begin{proof}
It is obvious using the equations \eqref{4-10}, \eqref{4-11} and \eqref{4-12}.
\end{proof}

\subsection{Involutes of slant helices:}

\begin{definition}
\label{4.8}Let $\alpha:I\subset \mathbb{R\rightarrow}G$ \ be an arc-lenghted
regular curve. Then the curve $x:I^{\ast}\subset \mathbb{R\rightarrow}G$ is
called the involute of the curve $\alpha$ if the tangent vector field of the
curve $\alpha$ is perpendicular to the tangent vector field of the curve $x.$
That is,%
\[
\left \langle T(s),T_{x}(s^{\ast})\right \rangle =0
\]
where $T$ and $T_{x}$ are the tangent vector fields of the curves $\alpha$ and
$x,$ respectively. Moreover $\left(  x,\alpha \right)  $ is called the
involute-evolute curve couple which are given by $\left(  I,\alpha \right)  $
and $\left(  I^{\ast},x\right)  $ coordinate neighbourhoods, respectively.
Then the distance between the curves $x$ and $\alpha$ are given by%
\[
d_{L}\left(  \alpha \left(  s\right)  ,x\left(  s\right)  \right)  =\left \vert
c-s\right \vert \text{, }c=\text{constant }\forall s\in I,
\]
\cite{struik}. We should remark that the parameter $s$ generally is not an
arc-length parameter of $x.$ So, we define the arc-length parameter of the
curve $x$ by
\[
s^{\ast}=\psi \left(  s\right)  =\int \limits_{0}^{s}\left \Vert \frac{dx\left(
s\right)  }{ds}\right \Vert ds
\]
where $\psi:I\longrightarrow I^{\ast}$ is a smooth function and holds the
following equality%
\begin{equation}
\psi^{\prime}\left(  s\right)  =\left(  c-s\right)  \varkappa \label{4-13}%
\end{equation}
for $s\in I.$
\end{definition}


\begin{theorem}
\label{4.9}Let $\alpha:I\subset \mathbb{R\rightarrow}G$ \ be an arc-lenghted
regular curve and $x$ be an involute of $\alpha$. Then $\alpha$ is a slant
helix in a three dimensional Lie group if and only if $x$ is a general helix.
\end{theorem}


\begin{proof}
Let $x$ be the involute of $\alpha$, then we have%
\[
x(s)=\alpha(s)+\left(  c-s\right)  T\left(  s\right)  ,\text{ }%
c=\text{constant.}%
\]
Let us derive both side with respect to $s$%
\begin{align*}
\frac{d\beta}{ds^{\ast}}\frac{ds^{\ast}}{ds} &  =\left(  c-s\right)
\overset{\cdot}{T}(s),\\
T_{x}\left(  s^{\ast}\right)  \frac{ds^{\ast}}{ds} &  =\left(  c-s\right)
\varkappa N,
\end{align*}
where $s$ and $s^{\ast}$ are arc-parameters of $\alpha$ and $x$, respectively.
Then we calculate as%
\[
\frac{ds^{\ast}}{ds}=\psi^{\prime}\left(  s\right)  =\left(  c-s\right)
\varkappa.
\]
and using this fact we can write%
\begin{equation}
T_{x}\left(  s^{\ast}\right)  =N.\label{4-14}%
\end{equation}
If we differentiate the last equation and use Frenet formulas then we obtain%
\begin{align*}
\varkappa_{x}N_{x}\left(  s^{\ast}\right)  \frac{ds^{\ast}}{ds} &
=\overset{\cdot}{N}=D_{T}N-\dfrac{1}{2}\left[  T,N\right] \\
\varkappa_{x}N_{x}\left(  s^{\ast}\right)  \varkappa &  =-\kappa T+\tau
B-\dfrac{1}{2}\left \langle \left[  T,N\right]  ,B\right \rangle B
\end{align*}
or with the help of the Proposition \ref{3.4}, we get%
\[
\varkappa_{x}N_{x}\left(  s^{\ast}\right)  =-T+HB
\]

where $\varkappa_{x}$ is the curvature of $x.$ Hence%
\[
\varkappa_{x}=\sqrt{1+H^{2}}%
\]
and%
\begin{equation}
N_{x}\left(  s^{\ast}\right)  =-\tfrac{1}{\sqrt{1+H^{2}}}T+\tfrac{H}%
{\sqrt{1+H^{2}}}B\label{4-15}%
\end{equation}
Then using the Eq.\eqref{4-14} and the Eq.\eqref{4-15} we have%
\begin{align}
B_{x}\left(  s^{\ast}\right)   & =T_{x}\left(  s^{\ast}\right)  \times
N_{x}\left(  s^{\ast}\right) \nonumber \\
& =\tfrac{H}{\sqrt{1+H^{2}}}T+\tfrac{1}{\sqrt{1+H^{2}}}B.\label{4-16}%
\end{align}
Using the differentiation of the last equation and the Proposition \ref{3.4},
this implies%
\[
\left(  \tau_{x}-\tau_{G_{x}}\right)  N_{x}\left(  s^{\ast}\right)
\frac{ds^{\ast}}{ds}=-\tfrac{H^{\prime}}{\left(  1+H^{2}\right)  ^{3/2}%
}T+\tfrac{HH^{\prime}}{\left(  1+H^{2}\right)  ^{3/2}}B
\]
or using the Eq.\eqref{4-13}, we have%
\[
\left(  \tau_{x}-\tau_{G_{x}}\right)  N_{\beta}\left(  s^{\ast}\right)
=-\tfrac{H^{\prime}}{\varkappa \left(  1+H^{2}\right)  ^{3/2}}T+\tfrac
{HH^{\prime}}{\varkappa \left(  1+H^{2}\right)  ^{3/2}}B
\]
where $\tau_{G_{x}}=\frac{1}{2}\left \langle \left[  T_{x},N_{x}\right]
,B_{x}\right \rangle .$ Thus we compute%
\[
\tau_{x}=\frac{H^{\shortmid}}{\varkappa \left(  1+H^{2}\right)  }+\tau_{G_{x}}%
\]
where $\tau_{x}$ is the torsion of $x.$ The we can easily see that
$\tfrac{\tau_{x}-\tau_{G_{x}}}{\varkappa_{x}}=\frac{H^{\shortmid}}%
{\varkappa \left(  1+H^{2}\right)  ^{3/2}}$\ is a constant function. In other
words, using the Theorem \ref{2.3} $x$ is a general helix.

Conversely, we assume that $x$ is a general helix then we can easily see that
$\alpha$ is a slant helix. These complete the proof.
\end{proof}


\begin{corollary}
\label{4.12}Let $\alpha:I\subset \mathbb{R\rightarrow}G$ \ be an arc-lenghted
regular curve and $\beta:I\subset \mathbb{R\rightarrow}S^{2}\subset
\mathfrak{g}$ be the tangent indicatrix of the curve $\alpha.$ If $\alpha$ is
a slant helix, then $\beta$ is one of the involutes of the curve $\alpha.$
\end{corollary}


\begin{proof}
It is obvious from the Theorem \ref{4.2} and the Theorem \ref{4.9}$.$
\end{proof}


\begin{corollary}
\label{4.13}Let $\alpha:I\subset \mathbb{R\rightarrow}G$ \ be an arc-lenghted
regular curve and $\delta:I\subset \mathbb{R\rightarrow}S^{2}\subset
\mathfrak{g}$ be the binormal indicatrix of the curve $\alpha.$ If $\alpha$ is
a slant helix, then $\delta$ is one of the involutes of the curve $\alpha.$
\end{corollary}


\begin{proof}
It is obvious from the Theorem \ref{4.7} and the Theorem \ref{4.9}$.$
\end{proof}

\begin{corollary}
Let $\alpha$\ be an arc-lenghted regular curve with the Frenet vector fields
$\left \{  T,N,B\right \}  $ in the Lie group $G$ and $x$ be the involute of the
curve $\alpha.$ Then $\tau_{G_{x}}=\tau_{G}$ for the curves $\alpha$ and $x.$
\end{corollary}

\begin{proof}
It is obvious using the equations \eqref{4-14}, \eqref{4-15} and \eqref{4-16}.
\end{proof}

\end{document}